%% file: derivedugkduality.tex
\documentclass[12pt]{amsart}

\usepackage{amssymb,amsmath,amsthm}
\usepackage[abbrev]{amsrefs}

\usepackage{graphicx}
\usepackage{braket,paralist}
\usepackage{bbold,eufrak,mathrsfs}
\usepackage[all,cmtip]{xy}
\usepackage{rotating}
\usepackage{arydshln}

\usepackage[table]{xcolor}
\usepackage{mathbbol}

\usepackage{multirow}

\usepackage{hyperref}
\usepackage{cleveref}

\theoremstyle{change}
\newtheorem{thm}[subsection]{Theorem}
\newtheorem{lemma}[subsection]{Lemma}
\newtheorem{prop}[subsection]{Proposition}
\newtheorem{cor}[subsection]{Corollary}
\newtheorem{definition}[subsection]{Definition}

\def\frametitle{}

\include{mymacros}

\def\wtKGp{{\wtK'}}
\def\wtKH{{\wtM}}

\def\XiGHp{{\Xi_{\fgg,\fhh'}}}
\def\XiHpH{{\Xi_{\fhh',\fhh}}}
\def\Wc{{W_\bC}}

\def\gU{{\sU}}
\def\gV{{\sV}}
\def\gT{{\sT}}

\begin{document}



\title[Derived functor, Dual pairs and $\cU(\fgg)^K$-actions]{Derived functor
  modules,\\ Dual pairs and $\cU(\fgg)^K$\mbox{-}actions} 

\author{Ma Jia Jun}

% \address{Department of Mathematics, National University of Singapore,
%   21 Lower Kent Ridge Road, Singapore 119077.}

\address{Department of Mathematics, 
Ben-Gurion University of the Negev,
P.O.B. 653, Be'er Sheva 84105,
Israel}

\email{hoxide@gmail.com}

\maketitle
\begin{abstract}
  Derived functors (or Zuckerman functors) play a very important role
  in the study of unitary representations of real reductive
  groups. These functors are usually applied on highest weight modules
  in the so-called good range and the theory is well-understood. On
  the other hand, there were several studies on the irreducibility and
  unitarizability, in which derived functors are applied to singular
  modules. See Enright et al. (Acta. Math. 1985), for example. In this
  article, we apply derived functors to certain modules arising from
  the formalism of local theta lifting, and investigate the
  irreducible sub-quotients of resulting modules. The key technique is to
  understand $\cU(\fgg)^K$ actions in the setting of a see-saw pair.  Our
  results strongly suggest that derived functor constructions are
  compatible with local theta lifting.

  % The well known theory of $A_\fqq(\lambda)$, is about derived functor
  % modules of highest weight modules in certain good range.  In this
  % paper, we apply derived functors (Zuckerman functors) on singular
  % representations. In the works of Enright (1983), Enright at. el. (1985), Frajria (1991),
  % Wallach (1994) and Wallach-Zhu(2004), they establish several
  % criteria on the irreducibility and unitarizablity of the resulting
  % modules. They also construct families of examples. By our main theorem,
  % Theorem~A, we show that some examples are related to theta lifts of
  % characters.  Our key observation is \Cref{lem:ugk} on the
  % $\cU(\fgg)^K$-actions in dual pairs. 

  % Motivated by a joint work with Loke, we also give examples beyond the case of theta lifts of characters, which
  % suggest that derived functor construction could be compatible with
  % theta correspondence in a much more general setting.
\end{abstract}

\section{Introduction}\label{sec:intro}
In this paper, we will study some small representations obtained by
applying derived functors on certain local theta lifts. The main objective
of this paper is to show that the resulting representations are still
theta lifts.  This project is motivated by \cite{WallachZhu2004}.

\subsection{}\label{ss:der}
For a Harish-Chandra pair $(\fgg,\rK)$, we denote by $\sC(\fgg,\rK)$ the
category of $(\fgg, \rK)$-modules (not necessarily admissible).  For a
real reductive group $\rG$, $(\fgg,\rK)$ is a Harish-Chandra pair,
where $\fgg= \Lie(\rG)_\bC$ and $\rK$ is a maximal compact subgroup of
$\rG$. Let $\rM$ be a subgroup of $\rK$.  Zuckerman functor
$\Gamma_{\fgg,\rM}^{\fgg,\rK}\colon \sC(\fgg,\rM)\to \sC(\fgg,\rK)$ is
the right adjoint functor of the forgetful functor
$\cF_{\fgg,\rK}^{\fgg,\rM}\colon \sC(\fgg,\rK)\to \sC(\fgg,\rM)$.
This functor is left exact and usually send a module in $\sC(\fgg,\rM)$
to zero.  On the other hand, its derived functors $R^j\Gamma_{\fgg,\rM}^{\fgg,\rK}$
construct interesting objects in $\sC(\fgg,\rK)$.

Using these derived functors, one can transfer representations between
different real forms of a complex reductive group as follows:  Let
$\fgg$ be the complex Lie algebra of a complex reductive group $\rG_\bC$.
Let $\sigma_1$ and $\sigma_2$ be two commuting involutions on $\fgg$.
They define two Harish-Chandra pairs $(\fgg,\rK_i)$ and the
corresponding real forms $\rG_i\in \rG_\bC$ such that $\fkk_i :=
\Lie(\rK_i)_\bC = \fgg^{\sigma_i}$. Therefore, every Lie algebra in the
following diagram is a symmetric subalgebra of the Lie algebra above
it.  
\def\ldsubseteq{\rotatebox{45}{$\subseteq$}} 
\def\rdsubseteq{\rotatebox{-45}{$\supseteq$}}
\[
\xymatrix@!=2pt {
  & \fgg \ar@{}[ld]_{\sigma_1}|{\ldsubseteq}\ar@{}[rd]^{\sigma_2}|{\rdsubseteq}&  \\
  \fkk_1 \ar@{}[rd]_{\sigma_2}|{\rdsubseteq}& & \fkk_2 \ar@{}[ld]^{\sigma_1}|{\ldsubseteq}\\
  & \fkk_1\cap \fkk_2 & }
\]

Let $\rM := \rK_1\cap \rK_2$. We define functor $\Gamma^j$ by
composing the forgetful functor and derived functors for each
non-negative integer $j$:
\[
\xymatrix@C=3em{ \Gamma^j : \sC(\fgg,\rK_1)
  \ar[r]^<>(0.5){\cF_{\fgg,\rK_1}^{\fgg,\rM}}& \sC(\fgg,\rM)
  \ar[r]^<>(0.5){\rR^j\Gamma_{\fgg,\rM}^{\fgg,\rK_2}}
  &\sC(\fgg,\rK_2).}
\]
We call $\Gamma^j$ the {\em transfer functor}, which transfers
$(\fgg,\rK_1)$-modules into $(\fgg,\rK_2)$-module.

% Let $\Sp(W)$ be the symplectic group of an real symplectic space.
% Let $(G,G')$ real reductive reductive dual pair Choose a maximal
% compact subgroup $\rU(W)$ in $\Sp(W)$ such that intersection with
% $G$ and $G'$ give their maximal compact subgroup $K$ and $K'$.  Let
% $\sY$ be the Fock model of the oscillator representation $\omega$ of
% $\wtSp(W)$ respect to this maximal compact subgroup where $\wtSp(W)$
% is the metaplectic cover of $\Sp(W)$.  For any subgroup $E$ of $G$
% such that $K_E := E\cap \rU$ is a maximal compact subgroup of $E$,
% let $\sR(\fee,\wtE;\sY)$ be the infinitesimal equivalent class of
% $(\fee,\wtK_E)$-module which can be realized as a quotient of
% $\sY$. All elements in $\sR(\fee,\wtK_E;\sY)$ are genuine
% representations of the double covering, i.e. the center act
% non-trivially.  Howe~\cite{Howe1989Tran} constructs a bijection
% $\theta \colon \sR(\fgg, \wtK;\sY) \to \sR(\fgg', \wtK'; \sY)$. This
% map is called theta lifting. By abuse notion, we also call its
% inverse $\theta$.

\subsection{}\label{sec:i.theta}
Let $G_\bC$ be a classical complex group. Let $G_1$ and $G_2$ be
subgroups of $G_\bC$ with commuting Cartan involutions as in
\Cref{ss:der}.  Let $(G_1,G'_1)$ and $(G_2, G'_2)$ be two real
reductive dual pairs such that $G'_1$ and $G'_2$ are real forms of a
complex group $G'_\bC$. We always denote by $\wtG$ certain double
covering of $G$. Let $\theta_i \colon \sR(\fgg',\wtK'_i;\sY_i) \to
\sR(\fgg,\wtK_i;\sY_i)$ be the theta lifting map from $\wtG'_i$ to
$\wtG_i$, where $\sR(\fgg,\wtK_i)$ is the set of infinitesimal
equivalent classes of admissible irreducible
$(\fgg,\wtK_i)$\mbox{-}modules in the domain of theta lifting
(c.f. \Cref{sec:ugkact}) and $\sY_i$ are fixed Fock-models of
oscillator representations.  Set $\rK_i := \wtK_i$, $M = K_1\cap K_2$
and $\rM := \wtK_1\cap \wtK_2 =\wtM$. We have
\begin{equation}\label{eq:gamma1}
\xymatrix{
  \Gamma^j := \rR^j\Gamma_{\fgg,\wtM}^{\fgg,\wtK_2}\circ
  \cF_{\fgg,\wtK_1}^{\fgg,\wtM} \colon \sC(\fgg,\wtK_1)\ar[r]&   \sC(\fgg,\wtM)
  \ar[r]& \sC(\fgg,\wtK_2).}
\end{equation}

We will exhibit some relationships between the transfer functors
$\Gamma^j$ and theta lifting maps $\theta_i$.  An optimistic guess is
that there exists certain operation ``?''  filling the gap and make
the diagram \eqref{eq:dig} commute. Here the inclusions in
\eqref{eq:dig} identify an infinitesimal equivalent class of
Harish-Chandra modules with an element in it.
\begin{equation}\label{eq:dig}
  \vcenter{
    \xymatrix@R=2em@C=1.5em{
      \sR(\fgg',\wtK'_1;\sY_1) \ar[r]^{\theta_1}\ar@{..>}[d]_{?}&  \sR(\fgg,\wtK_1;\sY_1) \ar@{^(->}[r]&\sC(\fgg,\wtK_1)
      \ar[d]^{\displaystyle\Gamma^j}\\
      \sR(\fgg',\wtK'_2;\sY_2) \ar[r]^{\theta_2}& \sR(\fgg,\wtK_2;\sY_2) \ar@{^(->}[r] & \sC(\fgg,\wtK_2)
    }}
\end{equation}
However, the actual relationships are more subtle. For example,
Theorem~\ref{thm:e1} shows that $\Gamma^j\theta_1(\rho)$ for $\rho\in
\sR(\fgg',\wtK'_1;\sY_1)$ could be reducible and its irreducible
components are theta lifts from different real reductive dual pairs.

\subsection{}\label{sec:thm}
We fix a non-trivial unitary character of $\bR$. Then we could fix the
Fock-models $\sY_i$ and the corresponding oscillator representations
in \eqref{eq:dig} (see \Cref{sec:Fock}).  We denote by $\Theta_i$ the
full theta lifting map of pair $(G_i,G'_i)$ (see \Cref{sec:Howe}).
Our main theorem is the following.
\begin{thmA}\label{thm:scaler}\label{THM:SCALER}
  Let $G'_1$ and $G'_2$ be two real forms of a classical complex Lie
  group $G'_\bC$ such that $(G_i, G'_i)$ form real reductive dual
  pairs.  Let $\rho_1$ and $\rho_2$ be characters of $\wtG'_1$ and
  $\wtG'_2$ respectively, $\tau_1$ be an irreducible
  $(\fkk_2,\wtK_1\cap \wtK_2)$-module and $\tau_2$ be an irreducible
  $\wtK_2$-module. Suppose that:
  \begin{enumerate}[(a)]
  \item \label{thmA.a} $\rho_1|_{\fgg'} = \rho_2|_{\fgg'}$;
  \item \label{thmA.b} $\tau_2$ occurs in $\Theta_2(\rho_2)$;
  \item \label{thmA.c} there exists a non-zero homomorphism $T\in
    \Hom_{\fkk_2, \wtK_1\cap\wtK_2}(\Theta_1(\rho_1), \tau_1)$, such
    that $\tau_2$ occurs in the image of $\Gamma^jT\colon \Gamma^j
    \Theta_1(\rho_1) \to R^j(\Gamma_{\fkk_2,\wtK_1\cap
      \wtK_2}^{\fkk_2,\wtK_2})\tau_1$.
  \end{enumerate}
  Then the two $(\fgg,\wtK_2)$-modules, $\Gamma^j \Theta_1(\rho_1)$
  and $\Theta_2(\rho_2)$, have isomorphic irreducible subquotients
  with the common $\wtK_2$-type $\tau_2$.
\end{thmA}

Roughly speaking above theorem says that the $K$-spectrums determine
the theta lifts of characters.  In applications, $\Theta_{1}(\rho_1)$
will be a direct sum of irreducible $(\fkk_2, \wtM)$-modules and the
precise $K$-spectrums of their derived functor modules are well
studied, see
\cite{Enright,Enright1985,Frajria1991,Wallach1994Trans,WallachZhu2004}
and \Cref{sec:ex}.

We highlight the key observation \Cref{lem:ugk}. It may have potential
usage beyond the case of lifts of characters.

This paper is a part of the author's Ph.D. Thesis.

% \begin{lemma}\label{lem:ugk}
%   Let $(G,G')$ and $(H,H')$ be a see-saw pair in $\Sp(W)$ such that
%   $H<G$ and $G'<H'$.  Let $\omega$ be an oscillator representation
%   of $\widetilde{\Sp(W)}$, then as a subalgebra of $\End_\bC(\sY)$,
%   \begin{equation}\label{eq:ugk}
%     \omega(\cU(\fgg)^H) = \omega(\cU(\fhh')^{G'}).
%   \end{equation}
%   Moreover for any $x\in \cU(\fgg)^H$ there is $y\in
%   \cU(\fhh')^{G'}$ such that $\omega(x) = \omega(y)$ such that the
%   choice could be made independent of real forms.
% \end{lemma}

% \begin{rmk}
%   \begin{enumerate}[1)]
%   \item It is a consequence of the construction of Fock model and
%     classical invariant theory
%     (c.f.~\cite{Adams2007}~\cite{Howe1989Rem}).  See
%     Appendix~\ref{sec:models} for detail.
%   \item The above lemma provides a tools to translate the
%     calculation of $\cU(\fgg)^K$ action on one side to the
%     $\cU(\fhh')^{G'}$ actions on the other side.
%     Theorem~\ref{thm:scaler} hold because $\cU(\fgg)^{K}$ is
%     calculable in the scaler $K$-type case.
%   \item In Lee, Nishiyama and Wachi's paper \cite{LeeNishiyama2008},
%     they also observed \eqref{eq:ugk} and were using it to study a
%     generalization of Capelli identity.
%   \end{enumerate}
% \end{rmk}

\subsection*{Notation}
In this paper, little Greek letters, for example $\tau$, denote the
infinitesimal equivalent classes of $(\fgg,\rK)$-modules and $V_\tau$
denotes a realization of $\tau$ on a vector space.  Moreover, $\tau$
also denotes the maps from $\cU(\fgg)$ and $\rK$ to $\End_\bC(V_\tau)$.
We will not distinguish representations and their isomorphism classes.

When there is no confusion, we omit the symplectic space $W$ in
$\Sp(W)$ and denote it by $\Sp$.  For any subgroup $G < \Sp$, $\wtG$
denotes the inverse image of $G$ in the metaplectic group $\wtSp$. We
always write $\fgg$ for the complexification of $\Lie(G)$.  Let $K$
and $K'$ denote the maximal compact subgroups of $G$ and $G'$. In this
paper, a complex Lie subalgebra $\fhh\subset \fgg$ would always be
stable under the fixed Cartan involution. Let $H$ denote the maximal
Lie subgroup of $G$ such that $\Lie(H)_\bC = \fhh$.  Therefore $M :=
H\cap K$ is a maximal compact subgroup of $H$.  Every real reductive
dual pair $(G,G')$ corresponds to a complex reductive dual pair
$(G_\bC,G'_\bC)$ by base change.  We always denote by $V^G$ the
subspace of $G$-invariants in a $G$-module $V$. In particular,
$\cU(\fgg)^H$ is the subalgebra of $H$-invariants in the universal
enveloping subalgebra $\cU(\fgg)$.  In the setting of transfer, $H$ is
always a symmetric subgroup of $G_1$ and $M = K_1\cap H = K_1\cap K_2
$ is a maximal compact subgroup of $H$.

% \red{We sometimes omite the
% forgetful functor from $\sC(\fgg,K)$ to $\sC(\fhh,M)$.}

% \subsection*{Acknowledgment}
% This paper is a part of the author's Ph.D. thesis. I would like to
% thanks my advisor Prof. Chen-bo Zhu for proposing this interesting
% topic to me. I would like thanks Prof. Zhu, Hung Yean Loke and Kyo
% Nishiyama for their stimulating conversations.

\section{Theta lifting, see-saw pair and joint $\cU(\fgg)^K$\mbox{-}actions}\label{sec:ugkact}

\subsection{}
We have following generalization of Howe's construction of
maximal quotient.
\begin{definition}
  Let $\rG$ be a real reductive group. Let $\rK$ and $\rH$ be
  subgroups of $\rG$ and let $\rM$ be a subgroup of $\rK\cap \rH$.  Let $V$
  be a $(\fgg,\rK)$-module and $U$ be an irreducible admissible 
  $(\fhh,\rM)$-module.  Define
  \[ \Omega_{V,U} = V / \cN_{V,U}, \quad \text{where}\quad \cN_{V,U} =
  \bigcap_{T\in \Hom_{\fhh,\rM}(V,U)}\Ker(T).
  \]
\end{definition}
Clearly $\Omega_{V,U}$ has joint actions of $\rH$ and
$\cU(\fgg)^{\rH}$.  It is known to experts (see \cite[Lemma~2.III.4]{MVW1987} or \cite[Section~2.3.3]{MaT}) that
\begin{equation}\label{eq:tensor}
  \Omega_{V,U} \cong U\otimes U',
\end{equation}
where $\cU(\fgg)^\rH$ acts on $U'\cong
\Hom_{\fhh,\rM}(U,\Omega_{V,U})$.

\smallskip

The key property of $\Omega_{V,U}$ is the following equation:
\begin{equation*}%\label{eq:maxquot}
  \Hom_{\fhh,M}(V,U) \cong \Hom_{\fhh,\rM}(\Omega_{V,U},U).
\end{equation*}
This equation leads to the well known see-saw pair argument due to
Kudla (c.f. \Cref{ssec:seesaw}).

\subsection{}\label{sec:Howe}
Now let $(G,G')$ be a real reductive dual pair in $\Sp$.  Let $\rU$ be
a fixed maximal compact subgroup of $\Sp$. Set $\rG:=\wtSp$, $\rK :=
\widetilde{\rU}$, $\rH:=\wtG$, $\rM := \wtK$, $V:=\sY$ the Fock model
and $U := V_\rho$ an irreducible admissible $(\fgg,\wtK)$-module. Set
$\Theta(\rho):=U'$. Then \eqref{eq:tensor} gives
\[
\Omega_{\sY, V_\rho} \cong V_{\rho} \otimes \Theta(\rho).
\]
Howe~\cite{Howe1989Tran} shows that when $\Theta(\rho)$ is non-zero,
it is a finite length $(\fgg',\wtK')$-module and it has a unique
irreducible quotient $\theta(\rho)$.  We call $\Theta(\rho)$ the
\em{full (local) theta lift}\footnote{It is called the maximal Howe
  quotient or big theta lift in some literatures.} of $\rho$ and
$\theta(\rho)$ the \em{(local) theta lift} of $\rho$.

Denote by $\sR(\fgg, \wtK;\sY)$ the set of infinitesimal equivalent
classes of irreducible admissible $(\fgg,\wtK)$-modules such that
$\Theta(\rho) \neq 0$, i.e. which could be realize as a quotient of
$\sY$.  Then $\rho\mapsto \theta(\rho)$ induces a one to one
correspondence
\[
\xymatrix{ \sR(\fgg,
  \wtK;\sY)\ar@{<->}[r]^<>(.5){\theta}&\sR(\fgg',\wtK';\sY), }
\]
which is called {\em (local) theta correspondence} or {\em Howe correspondence}.
The roles of $G$ and $G'$ are symmetric. By abusing notation, we also
denote by $\Theta$ the lifting from $\wtG'$ to $\wtG$.

\subsection{} \label{sec:Harish} Let $V$ be an admissible
$(\fgg,K)$-module and $U$ be an irreducible $K$-module. Set $\rG :=G$,
$\rH:=\rM:=K$.  Then $\Omega_{V,U}$ is the $U$\mbox{-}isotypic
component of $V$.  There is a well known result of
Harish\mbox{-}Chandra~\cite{Lep} that irreducible constituents
containing the $K$\mbox{-}type $U$ in the $(\fgg,K)$-module $V$ are in
one to one correspondence to irreducible consistuents of the
$\cU(\fgg)^K$\mbox{-}module $U' := \Hom_{K}(U,\Omega_{V,U})$.  Note
that the right $\cU(\fgg)^K$-module, $\Hom_{K}(V,U) =
\Hom_{K}(\Omega_{V,U},U)$, is the dual of the finite dimensional left
$\cU(\fgg)^K$-module $U'$. So we will only consider $\Hom_{K}(V,U)$ in
the rest of this paper.

\subsection{See-saw pairs}\label{ssec:seesaw}
A see-saw pair is a pair of reductive dual pairs $(G,G')$ and $(H,H')$ in
$\Sp$ such that $H < G$ and $H'> G'$. The relationship
between these groups is given in the following diagram.
\def\dsubset{{\rotatebox{90}{$\subseteq$}}}
\[
\xymatrix@!=1em{G \ar@{-}[dr] \ar@{}[d]|{\dsubset} & H' \ar@{}[d]|{\dsubset}\\
  H\ar@{-}[ur] & G' }
\]

\begin{lemma}\label{lem:seesaw}
 Let $\tau\in \sR(\fhh,\wtKH;\sY)$ and $\rho\in
  \sR(\fgg',\wtKGp;\sY)$. Then
  \begin{equation}\label{eq:isohomsp}
    \Hom_{\fhh, \wtKH}(\Theta(\rho),V_\tau) \cong
    \Hom_{(\fhh,\wtKH)\times (\fgg',\wtKGp)}(\sY, V_{\tau}\otimes
    V_{\rho})
    \cong \Hom_{\fgg',\wtKGp}(\Theta(\tau),V_{\rho})
  \end{equation}
  Here $\Hom_{\fhh,\wtKH}(\Theta(\rho),V_\tau)$ is a right
  $\cU(\fgg)^{\wtH}$-module and
  $\Hom_{\fgg',\wtKGp}(\Theta(\tau),V_{\rho})$ is a right
  $\cU(\fhh')^{\wtG'}$-module by pre-composition. The
  first isomorphism in \eqref{eq:isohomsp} is\/
  $\cU(\fgg)^{\wtH}$-equivariant and the second isomorphism is
  $\cU(\fhh')^{\wtG'}$\mbox{-}equivariant.
\end{lemma}
\proof It is clear from the following calculation.
\begin{equation*}%\label{eq:seesawhom}
  \begin{split}
    \Hom_{\fhh,\wtKH}(\Theta(\rho), V_\tau) \cong &
    \Hom_{(\fhh,\wtKH)\times
      (\fgg',\wtKGp)}(\Theta(\rho)\otimes V_\rho, V_\tau\otimes V_\rho)\\
    \cong & \Hom_{(\fhh,\wtKH)\times
      (\fgg',\wtKGp)} (\Omega_{\sY,V_\rho}, V_\tau\otimes V_\rho)\\
    \cong & \Hom_{(\fhh,\wtKH)\times (\fgg',\wtKGp)} (\sY,
    V_\tau\otimes V_\rho) \cong \Hom_{\fgg',\wtKGp}(\Theta(\tau),
    V_\rho).\qed
  \end{split}
\end{equation*}

The above proof is formal and is actually valid for any category of
representations and any ``see-saw pairs'' of mutually commuting
subgroups when Schur's Lemma holds.

\subsection{}
In the case of local theta correspondence over real,
Lemma~\ref{lem:ugk} shows that the joint actions of $\cU(\fgg)^{\wtH}$
and $\cU(\fhh')^{\wtG'}$ on $\sY$ factor through the same subalgebra
of $\End_{\bC}(\sY)$. Therefore \eqref{eq:isohomsp} links the
$\cU(\fgg)^{\wtH}$ and $\cU(\fhh')^{\wtG'}$\mbox{-}actions on its two
sides.  We would like to point out that
Lee-Nishiyama-Wachi~\cite{LeeNishiyama2008} have obtained the lemma
when $H$ and $G'$ are both compact in a study of Capelli identities.
\Cref{lem:ugk} is a generalization of the correspondence of
infinitesimal characters~\cite{Przebinda1996inf}, where we set $(H,H')
:= (G,G')$.

\begin{lemma}\label{lem:ugk}
  Let $\omega$ be an oscillator representation of $\wtSp$ with Fock
  model $\sY$. Let $(G,G')$ and $(H,H')$ form a see-saw pair in $\Sp$
  such that $H<G$ and $G'<H'$.  Then
  \begin{enumerate}[(i)]
  \item as subalgebras of $\cU(\fgg)$ and $\cU(\fhh')$,
    \[
    \cU(\fgg)^{H_\bC} = \cU(\fgg)^H = \cU(\fgg)^{\wtH} \quad
    \text{and}\quad  \cU(\fhh')^{G'_\bC} = \cU(\fhh')^{G'} =
    \cU(\fhh')^{\wtG'} \quad \text{respectively};
    \]
  \item as subalgebras of $\End_\bC(\sY)$,
    \begin{equation*}%\label{eq:ugk}
      \omega(\cU(\fgg)^{\wtH}) = \omega(\cU(\fgg)^{H_\bC}) =
      \omega(\cU(\fhh')^{G'_\bC})= \omega(\cU(\fhh')^{\wtG'}).
    \end{equation*}
  \end{enumerate}
  In particular, here exist a map (may not be unique and may not be an
  algebra homomorphism) $\XiGHp\colon\cU(\fgg)^{H_\bC}\to
  \cU(\fhh')^{G'_\bC}$ such that $\omega(x) = \omega(\XiGHp(x))$.
  Moreover, $\XiGHp$ could be defined only depending on the complex dual
  pair, but independent of real forms and 
  $\omega$. 
  % We fix such kind of $\XiGHp$ for every (complex)
  % see-saw pair $((G_\bC,G'_\bC),(H_\bC,H'_\bC))$. 
\end{lemma}
\begin{proof}[Sketch of the proof (see {\cite[Section~2.3.4]{MaT}} for
  details)]
  % This is evident from~\cite{Howe1989Rem}.
  The actions of the double coverings on its Lie algebra
  factor through the linear group, $\cU(\fgg)^{\wtH} = \cU(\fgg)^H$
  and $\cU(\fhh')^{\wtG'} = \cU(\fhh')^{G'}$.  By the classification
  of real reductive dual pairs, $H$ meets all the connected components
  of $H_\bC$. Hence, $\cU(\fgg)^H = \cU(\fgg)^{H_\bC}$ and (i)
  follows.

  To prove (ii), it suffices to show $\omega(\cU(\fgg)^{H_\bC}) =
  \omega(\cU(\fhh')^{G'_\bC})$.  Let $\Wc$ be the complex symplectic
  space defining $\fsp$ and let $\fee = \Wc \oplus \bC$ be the corresponding
  Heisenberg Lie algebra.  Under the notation in~\cite{Howe1989Rem},
  let $\End^\circ$ be the image of $\cU(\fee)$ in $\End(\sY)$.
  Howe~\cite[Theorem~7]{Howe1989Rem} shows that  
  \[
  \omega(\cU(\fgg)) = (\End^\circ)^{G'_\bC} \quad \text{and}\quad
  \omega(\cU(\fhh')) = (\End^\circ)^{H_\bC}
  \]
  by the classical invariant theory.
  Therefore
  \[
  \omega(\cU(\fgg)^{H_\bC}) = \omega(\cU(\fgg))^{H_\bC} =
  ((\End^\circ)^{G'_\bC})^{H_\bC} = ((\End^\circ)^{H_\bC})^{G'_\bC} =
  \omega(\cU(\fhh')^{G'_\bC}).
  \]

  Note that $\End^\circ$ could be realized as an abstract quantum
  algebra which is only depending on $\Wc$.
  This ensures that $\XiGHp$ could be defined
  independent of real forms and oscillator representations. See
  \Cref{sec:Fock}.
\end{proof}

\subsection{Remarks on Fock-models}\label{sec:Fock}
The materials in the section is due to Kudla. See
\cite[Section~2]{Adams2007} or \cite[Section~2.3.4]{MaT}. We present a
sketch here for completeness.

\def\Wc{W_\bC} Let $W_\bC$ be a complex symplectic space with
symplectic form $\inn{}{}$.  Fix a primitive fourth root of unity $i =
\sqrt{-1}$. We hence fixed a unitary character $\varphi(r) = \exp(ir)$
of $\bR$.

Define the quantum algebra $\Omega(\Wc) := \cT/\cI$ where $\cT$ is the
tensor algebra of $\Wc$ and $\cI$ is the two side ideal in $\cT$
generated by $ \set{v\otimes w-w\otimes v - i\inn{v}{w}|v,w\in
  W_\bC}$.  The algebra $\Omega(\Wc)$ has a natural filtration induced by the
natural filtration on $\cT$.  Let $\Omega_j(\Wc)$ be the space of
elements with degree less than or equal to $j$. Let $\fee = \Wc\oplus \bC$
be the Heisenberg Lie algebra of $\Wc$ with center $\bC$.  Then
there is a unique isomorphism of Lie algebras,
\[\xymatrix{\omega_\bC\colon \sp(W_\bC)\ltimes \fee \; \ar@{^(->>}[r]&
  \Omega_2(\Wc)}\]
extending the natrual map $\Wc \to \Omega_1(\Wc)$. We extend above
map to universal enveloping algebras and still call it $\omega_\bC$.

Now fix a real form $W$ of $\Wc$, i.e. $\Wc = W\otimes_\bR \bC$.  
A totally complex polarization of $\Wc$ with respect to $W$ is a
decomposition $\Wc= X\oplus Y$ such that $X$ and $Y$ are maximal
isotropic $\bC$\mbox{-}subspaces in $\Wc$ and $X \cap W = 0$. 
% There essentially $\dim_{\bC} W_\bC/2 +1$ choice of $X$ upto
% $\Sp(W)$-conjugation. But the issue is not important here. Under
% exactly one of the choice,
The Fock-model of the oscillator representation of
$\widetilde{\Sp(W)}$ associated with the central character $\varphi$
of the Heisenberg group is given by
\begin{equation*}%\label{eq:Focksp}
  \sY := \Omega(\Wc)/\Omega(\Wc)X.
\end{equation*}
Here $\sp(W)_\bC = \sp(W_\bC)$ acts on $\sY$ by compositing $\omega_\bC$
and the left multiplication. Moreover, $\sY\cong \bC[X]$ as vector
spaces and $\Omega(\Wc) \cong \End^\circ$ in Howe's
picture~\cite{Howe1989Rem}.

Let $(G_\bC,G'_\bC)$ be a complex dual pair in $\Sp(\Wc)$. Recall that
$\fgg := \Lie(G_\bC)$ and $\fgg' := \Lie(G'_\bC)$.
Howe~\cite{Howe1989Rem}'s result is rephrased into
\begin{equation*}%\label{eq:cinvg1}
  \omega_\bC(\cU(\fgg)) = \Omega(\Wc)^{G'_\bC} \quad \text{and}\quad
  \omega_\bC(\cU(\fgg')) = \Omega(\Wc)^{G_\bC}.
\end{equation*}

The following lemma is a rephrase of the equation $(2.4)$ in
\cite{Howe1989Tran} and one can check it case by case according to the
classification of real reductive dual pairs.
\begin{lemma}
  %Fix a complex reductive dual pair $(G_\bC,G'_\bC)$ in $\Sp(\Wc)$.  
  Let $G$ and $G'$ be real forms of $G_\bC$ and $G'_\bC$ respectively
  such that $(G,G')$ is a real reductive dual pair. Then there is a
  real form $W$ of $\Wc$ such that
  \[
  \Lie(G) = \sp(W)\cap \fgg, \quad \Lie(G') = \sp(W)\cap \fgg',
  \]
  where $\fgg$, $\fgg'$, $\Lie(G)$, $\Lie(G')$ and $\sp(W)$ are
  considered to be Lie subalgebras of $\sp(\Wc)$. \qed
\end{lemma}
Note that the oscillator representation $\omega$ of $\sp(W)$ acts on
the Fock model $\sY$ factor through $\omega_\bC$. By the
commutative diagram \eqref{eq:models},
$\XiGHp(x)$ in \Cref{lem:ugk} could be made independent of real forms
via $\omega_\bC$.
\begin{equation}\label{eq:models}
  \vcenter{
    \xymatrix{
      \fgg\oplus \fgg' \ar@{^(->}[r] & \sp(\Wc) \ar[r]^{\omega_\bC}\ar[dr] &
      \Omega(\Wc)\ar[d]^{\text{left multiplication}} \\
      \Lie(G)\oplus \Lie(G')\ar@{^(->}[u]\ar@{^(->}[r]&\sp(W) \ar@{^(->}[u]\ar[r]^{\omega}& \End_\bC(\sY)
    }
  }
\end{equation}

\section{A theorem of Helgason and its consequences}
The following result of Helgason~\cite{Helgason1964Fund} is crucial for
us.
\begin{thm}\label{thm:scalerk}
  Let $\rG$ be a real reductive group such that all simple factors of
  $\fgg$ are classical Lie algebras.  Let $\rH$ be a symmetric
  subgroup of $\rG$: We suppose that there is an involution $\sigma$
  on $\Lie(\rG)$ such that $\rH$ is a subgroup of $\rG$ with Lie
  algebra $\Lie(\rG)^\sigma$ and meet all the connected components of
  $\rG$.  Let $\cZ(\fgg) = \cU(\fgg)^{\rG}$ be the subalgebra of
  $\rG$-invariants in $\cU(\fgg)$.  For a one-dimensional
  representation $\rho$ of $\rH$, let $\Ann_{\cU(\fhh)}(\rho)$ be the annihilator ideal of $\rho$ in
  $\cU(\fhh)$.  Then the natural map
  \[
  \xymatrix{ \cZ(\fgg) \ar[r]&
    \cU(\fgg)^\rH/(\Ann_{\cU(\fhh)}(\rho)\cU(\fgg)\cap\cU(\fgg)^\rH) }
  \]
  is surjective. \qed
\end{thm}

Our $\cZ(\fgg)$ can be smaller than the center of
$\cU(\fgg)$. For example, consider $\cU(\fso(2n))^{\rO(2n)}$. This is exactly
the situation we will encounter in local theta
correspondence. \Cref{thm:scalerk} may not hold if $\fgg$ has some
exceptional simple factors~\cite{Helgason1992}.

The above theorem is a small alteration of Helgason's original
version.  Helgason~\cite{Helgason1964Fund} treats the case that $\rH$
is a maximal compact subgroup of $\rG$ with $\rho$ trivial. Later
Shimura extends it to non-trivial $\rho$ 
(c.f. \cite[Theorem~2.4]{Shimura1990}).

Shimura's version and a Weyl's ``Unitary Trick'' (see for example
\cite{Lepowsky76}) implies above theorem. We omit the routine proof,
see \cite[Section~3.A]{MaT} for details.

\subsection{}
Following the argument in~\cite{Zhu2003},
% we have a easy but important consequence of \Cref{thm:scalerk}.
% \begin{lemma}\label{lem:scaler1}
%   In the setting of \Cref{thm:scalerk}, let $V$ be a
%   $(\fgg,\rK)$-module and $\rho$ be a character of $(\fhh,\rM)$.
%   Then the $\cU(\fgg)^\rH$-actions on $\Hom_{\fhh,\rM}(V,V_\rho)$
%   and, therefore, on $\Omega_{V,V_\rho}$ are determined by the
%   $\cZ(\fgg)$-actions on $V$ and the annihilator ideal $\cJ_\rho:=
%   \Ann_{\cU(\fhh)}(\rho)$ of $\rho$.
% \end{lemma}
% \begin{proof}
%   Let $0\neq T\in \Hom_{\fhh,\rM}(V,V_\rho)$.  For any $x\in
%   \cU(\fgg)^{\rH}$, choose $z\in \cZ(\fgg)$ such that $x-z = ju$ for
%   some $j\in \cJ_\rho$ and $u\in \cU(\fgg)$.  Then $T (x-z) = Tj u
%   =j T u = 0$. Therefore $Tx = Tz$.
% \end{proof}
% \subsection{}
we combine \Cref{thm:scalerk} and the see-saw pair
argument (\Cref{lem:seesaw}).  Let $(G,G')$ and $(H,H')$ form a see-saw
pair such that $H$ is a symmetric subgroup of $G$. By the
classification of real reductive dual pairs, $G'$ is automatically a
symmetric subgroup of~$H'$ and satisfies the condition in
\Cref{thm:scalerk}.

\begin{lemma}\label{thm:ugkcor}
  Let $\rho \in \cR(\fgg',\wtKGp;\sY)$ be a character of $\fgg'$ and
  let $\tau \in \cR(\fhh, \wtKH;\sY)$.  Then $\cU(\fgg)^{\wtH}$ acts
  on $\Hom_{\fhh,\wtKH}(\Theta(\rho), V_\tau)$ via a character.  This
  character is determined by the character $\chi_\tau$ of $\cZ(\fhh)$
  acting on $\tau$ and the annihilator ideal
  $\Ann_{\cU(\fgg')}(\rho)$.

  More precisely, there exists a map 
  \[\xi\colon \cU(\fgg)^{\wtH} =
  \cU(\fgg)^{H_\bC} \longrightarrow \cZ(\fhh) \quad
  \text{such that}\quad  T\circ \omega(x) = \chi_\tau(\xi(x)) T
\]
for $T\in \Hom_{\fhh, \wtKH}(\Theta(\rho),V_\tau)$. Moreover, $\xi$
could be defined only depending on $\Ann_{\cU(\fgg')}(\rho)$, but
independent of real forms and $\omega$.
\end{lemma}
\begin{proof} By Lemma~\ref{lem:ugk}, for $x \in \cug^{\wtH}$, let $x' :=
  \XiGHp(x) \in \cU(\fhh')^{\wtG'}$. We have $\omega(x) = \omega(x')$ by
  the definition of $\XiGHp$.  Choose $z' \in \czz(\fhh')$ such that
  $x'-z'=a'u'$ with $a'\in\Ann_{\cU(\fgg')}(\rho)$ and $u'\in
  \cU(\fhh')$ by Lemma~\ref{thm:scalerk}.  Let $z := \XiHpH(z')\in
  \czz(\fhh)$. Therefore $\omega(z) = \omega(z')$ by
  Lemma~\ref{lem:ugk} again.

  For any $0\neq T\in \Hom_{(\fhh,\wtKH)\times(\fgg',\wtKGp)}(\sY,
  V_\tau \otimes V_\rho)$,
  \[
  \begin{split}
    T\circ\omega(x) =& T\circ\omega(x') = T\circ \omega(x'-z') +T\circ
    \omega(z') \\
    =& T \circ \omega(a'u') + T\circ \omega(z)
    = \rho(a') \circ T \circ \omega(u') + \tau(z) \circ T\\
    =& \chi_{\tau}(z) T.
  \end{split}
  \]
  Define $\xi(x) := z$.  Now the Lemma follows from the
  $\cU(\fgg)^{\wtH}$-module isomorphism
  \eqref{eq:isohomsp}:
  \[ \Hom_{\fhh,\wtKH}(\Theta(\rho), V_\tau) \cong
   \Hom_{(\fhh,\wtKH)\times(\fgg',\wtKGp)}(\sY, V_\tau \otimes
   V_\rho)\]
  
  Since $\XiGHp$,
  $\XiHpH$ and $z'$ could be defined independent of real forms and
  $\omega$, so dose $\xi$. This finish the proof.
\end{proof}

\section{Proof of Theorem~\ref{thm:scaler}}\label{sec:trans}
\subsection{} We recall some basic facts from the
construction of derived functor modules (see for example,
%\cite[Section~I.8]{BorelWallach2000}, 
\cite[Chapter 6]{Wallach1988}% or\cite{NV}
). We retain the notation in~\Cref{ss:der} where $V\in
\sC(\fgg,\rM)$.

As $\rK$-modules,
\[
\left.\rR^j(\Gamma_{\fgg,\rM}^{\fgg,\rK}) V\right|_{\rK} =
\rR^j(\Gamma_{\fkk,\rM}^{\fkk,\rK}) V.
\]
Hence we also denote $\rR^j(\Gamma_{\fkk,\rM}^{\fkk,\rK})$ by $\Gamma^j$.

For each $x\in \cU(\fgg)^\rK$, it gives a $(\fkk,\rM)$-map $x\colon
V\to V$.  Then $x$ acts on $\Gamma^jV$ by $\Gamma^j x\colon \Gamma^j
V\to \Gamma^j V$. Hence we have following lemma.
\begin{lemma}\label{lem:derugkact}
  Let $V\in \sC(\fgg,\rM)$ and $U\in \sC(\fkk,\rM)$.  Let $T\in
  \Hom_{\fkk,\rM}(V,U)$.  Suppose that $T\circ x = \chi(x) T$ for a
  charactor $\chi\colon \cU(\fgg)^{\rK} \to \bC$.  Then
  \begin{enumerate}[(i)]
  \item $\cU(\fgg)^{\rK}$ acts on $\Gamma^jT\in
    \Hom_{\rK}(\Gamma^jV,\Gamma^jU)$ via that character $\chi$;
  \item in particular, if $\cZ(\fgg)$ acts on $V$ via a character, then
    $\cZ(\fgg)$ acts on $\Gamma^j V$ via the same character.\qed
  \end{enumerate}
\end{lemma}
  
Suppose $V$ is $\cZ$-finite where $\cZ$ is the center of $\cU(\fgg)$.
For example $V:=\Theta_1(\rho_1)$ in the next section.  By
\Cref{lem:derugkact} (ii), $\Gamma^jV$ is also $\cZ$-finite. Hence,
every finite dimensional $\rK$-invariant subspace in $\Gamma^jV$ will
generate an admissible $(\fgg,\rK)$-module
(c.f.~\cite{WallachZhu2004}).  It is also clear that
$\Ann_{\cU(\fgg)}(\Gamma^jV) \supset \Ann_{\cU(\fgg)}(V)$. So an
irreducible subquotient of $\Gamma^jV$ has Gelfand-Kirillov dimension
less than or equal to that of $V$.

\subsection{}
We retain the notation in \Cref{sec:thm}.
\begin{proof}[Proof of Theorem~\ref{thm:scaler}]
  We can fix a complex dual pair $(G_\bC,G'_\bC)$ sitting in
  $\Sp(W_\bC)$ such that $G_i \subset G_\bC$ and $G'_i \subset G'_\bC$
  with commuting Cartan involutions\footnote{These Cartan involutions
    of $G_i$ and $G'_i$ can be obtained by restricting a pair of
    commutating Cartan involutions of $\Sp(W_\bC)$}.

  Let $\fhh := \fkk_2$ and $H$ be the maximal subgroup in $G_1$ with
  Lie algebra $\fhh\cap \Lie(G_1)$. Now $H$ is the member of a dual
  pair $(H,H')$ and $M := H \cap K_2 = K_1\cap K_2$ is a maximal compact
  subgroup of $H$ by the classification of real reducitive dual pairs.
  % meet all the connected component of
  % $G_\bC$.

  % Set $M:=\wtK_1\cap \wtK_2$.
  Let $V_{\tau_1}$ be a $(\fkk_2,\wtM)$-module of type $\tau_1$. Let
  \[0\neq T \in \Hom_{\fkk_2,\wtM}(\Theta_1(\rho_1),V_{\tau_1})
  \] be the map in assumption \eqref{thmA.c}.  Since $\tau_2$ occur in
  the image of $\Gamma^jT$, we fix an irreducible $\wtK_2$-submodule
  $U_{\tau_2}$ of type $\tau_2$ in $\Gamma^j\Theta(\rho_1)$ such that
  $(\Gamma^jT)(U_{\tau_2}) \neq 0$. Let
  \[
  \gU := \cU(\fgg)U_{\tau_2} \subset \Gamma^j \Theta_1(\rho_1)
  \]
  be the admissible $(\fgg,\wtK_2)$-submodule of $\Gamma^jV$ generated
  by $U_{\tau_2}$.  Let $\gV$ be the image of $U_{\tau_2}$ under
  $\Gamma^jT$.  We have a right $\cU(\fgg)^{\wtK_2}$-module homomorphism
  \[
  \xymatrix{
    \Hom_{\wtK_2}(\Gamma^j\Theta_1(\rho_1),\Gamma^jV_{\tau_1})\ar[r]&
    \Hom_{\wtK_2}(\gU,\gV)}
  \]
  by pre-composite the restriction and post-composite the
  projection.  Denote by $\gT$ the image of $\Gamma^jT$ under above
  homomorphism.

  Notice that, as subalgebras in $\cU(\sp)$, $ \cU(\fgg)^{\wtH} =
  \cU(\fgg)^H = \cU(\fgg)^{K_2} = \cU(\fgg)^{\wtK_2}$ and $\cZ(\fhh) =
  \cU(\fkk_2)^{\wtH} = \cU(\fkk_2)^H=\cU(\fkk_2)^{K_2} = \cZ(\fkk_2)$.

 % \begin{equation*}%\label{eq:realformeqs}
 %    \begin{split}
 %      \cU(\fgg)^{\wtH} =& \cU(\fgg)^H = \cU(\fgg)^{K_2} =
 %      \cU(\fgg)^{\wtK_2} \quad \text{and}\\
 %      \cZ(\fhh) =& \cU(\fkk_2)^{\wtH} =
 %      \cU(\fkk_2)^H=\cU(\fkk_2)^{K_2} = \cZ(\fkk_2).
 %    \end{split}
 %  \end{equation*}

  Since $V_{\tau_1}$ is an irreducible $(\fkk_2,\wtM)$-module and
  $\gV$ is an irreducible $\wtK_2$-submodule, $\cZ(\fkk_2)$ act on
  them by characters.  On the other hand, $\sV$ is a submodule in
  $\Gamma^jV_{\tau_1}$, we conclude that $\cZ(\fkk_2)$ acts on
  $V_{\tau_2}$, $\gV$ and $\Gamma^jV_{\tau_1}$ by the same character
  (c.f. Lemma~\ref{lem:derugkact}~(ii)). 

  The
  assumption~\eqref{thmA.a} implies $\Ann_{\cU(\fgg')}\rho_1 =
  \Ann_{\cU(\fgg')}\rho_2$. Now $\cU(\fgg)^{\wtH}=\cU(\fgg)^{\wtK_2}$ act on
  $\Hom_{\fkk_2,\wtK_1}(\Theta_1(\rho_1),V_{\tau_1})$ and
  $\Hom_{\wtK_2}(\Theta_2(\rho_2), \gV)$ via the same character by
  \Cref{thm:ugkcor}. In particular, $\cU(\fgg)^{\wtK_2}$ acts on
  $\bC\gT$ via this character by \Cref{lem:derugkact}~(i).

  Hence $\gU$ has an irreducible quotient containing $U_{\tau_2}$ and
  it is isomorphic to the irreducible subquotient of
  $\Theta_2(\rho_2)$ containing $\tau_2$ by the discussion in
  \Cref{sec:Harish}.  This finished the proof.
\end{proof}

It is known to experts that $\Theta_2(\rho_2)$ is $\wtK_2$-multiplicity
free for any character $\rho_2$,
i.e. $\dim\Hom_{\wtK_2}(\Theta_2(\rho_2),\tau_2) \leq 1$ for every
$\wtK_2$-type $\tau_2$ (see for
example, \cite{ZhuHuang1997} or \cite[Section~2.3.6]{MaT}).  On the
other hand, the multiplicity of $\tau_2$ in $\Gamma^j\Theta_1(\rho_1)$
could be greater than one. For all examples in \Cref{sec:ex}, it is
the $\wtK_2$-isotypic component of several copies of a
$\wtK_2$-multiplicity free irreducible $(\fgg,\wtK_2)$\mbox{-}module.

\section{Examples}\label{sec:ex}
In this section, we give two types of examples. In these
examples, the full theta liftings are already irreducible. So we
replace $\Theta$ by $\theta$ when we apply \Cref{thm:scaler}.

\subsection{A decomposition of the derived functor module}
\label{sec:dec}
We retain notations in \Cref{ss:der}. %Let $V:=\theta_{\rho}$.
Suppose the $(\fgg,\rM)$-module $V$ is a direct sum of irreducible
unitarizable $(\fkk, \rM)$-modules, i.e.
\begin{equation*}%\label{eq:decH}
V = \bigoplus_{l\in L}V_l,
\end{equation*}
where the Harish-Chandra pair $(\fkk,\rM)$ comes from some real
reductive group $\rH$ and $V_l$ are irreducible unitarizable $(\fkk,\rM)$-modules.  In
fact, all examples in the next sections are in this case and they are
typical examples of discretely decomposable modules in the sense of
Kobayashi~\cite{Kob}.

\smallskip

Following Wallach-Zhu~\cite{WallachZhu2004}, we have a decomposition of
the $(\fgg,\rK)$-module $\Gamma^j V$ by Vogan-Zuckerman's
theory~\cite{VoganZuckerman1984}:
\begin{equation}\label{eq:decW}
  \Gamma^j V = \bigoplus_{\cW} \Gamma_{\cW} V.
\end{equation}
We describe the decomposition briefly. See \cite{WallachZhu2004} or
\cite[Section~2.4.2]{MaT} for the details of the construction.  Here we
fix a decomposition of $\bigwedge^j \fkk/\fmm$ into irreducible
$\rM$-submodules and $\cW$ runs over all its irreducible components. For
each $V_l$ let $\hat{\rK}_l$ be the set of $\rK$-type $\gamma$ such
that $V_l$ and $\gamma$ have the same infinitesimal characters and
central characters.  Then the $(\fgg,\rK)$-module
\[
\Gamma_{\cW} V := \bigoplus_{l\in L} \Gamma_{\cW} V_l :=
\bigoplus_{l\in L}\Hom_{\rM}(W,
\bigoplus_{\gamma\in \hat{\rK}_l} V_j\otimes \gamma^*)\otimes \gamma\]
as $\rK$-module, where $\Gamma_{\cW} V_l$ are $\rK$-submodules of
$\Gamma^j V_l$.  Most of the terms $\Gamma_{\cW} V$ in \eqref{eq:decW}
are zero, one can
determine the non-zero terms by Vogan-Zuckerman's
theory~\cite{VoganZuckerman1984}.

\subsection{Transfer of unitary highest weight modules }
\label{sec:exampleWZ}
To have unitary highest weight modules, the pair $(\fgg,K)$ should be
Hermitian symmetric.  We will study three families of examples where
$\fgg$ has root system of type $A$, $C$ and $D$ respectively.

\begin{table}[htbp]
  \centering\small\setlength{\tabcolsep}{2pt}
  \begin{tabular}{c|cc|c|c|c} 
    Type & $G$           & $G'_{p,q}$  & $H$                       & stable range   & $j(p,q)$           \\
    \hline
    $A$  & $\rU(n,n)$    & $\rU(p,q)$  & $\rU(r,s)\times \rU(s,r)$ & $n\geq p+q$    & $rs-(r-p)(s-q)$    \\
    $C$  & $\Sp(2n,\bR)$ & $\rO(p,q)$  & $\rU(r,s)$                & $n\geq p+q$    & $2(rs-(r-p)(s-q))$ \\
    $D$  & $\rO^*(2n)$   & $\rSp(p,q)$ & $\rU(r,s)$                & $n\geq 2(p+q)$ & $rs-(r-2p)(s-2q)$ 
  \end{tabular}
  \caption{Transfer of unitary highest weight modules: (Here $r+s=n$.)}
  \label{tab:tran1}
\end{table}

See \Cref{tab:tran1} for
notation. We fix a real form $G \subset G_\bC$ with Cartan involution
$\sigma_1$. We set $G_1 := G$. Fix an involution $\sigma_2$
commuting with $\sigma_1$ such that $H = G_\bC^{\sigma_2}\cap
G_1$. Let $G_2$ be the real form of $G_\bC$ such that $\sigma_2$ is
its Cartan involution.  Let $K_i = G_i^{\sigma_i}$, $M = K_1\cap K_2$
and define $\Gamma^j$ as in \eqref{eq:gamma1}.

In this setting, we have $\wtG = \wtG_1\cong \wtG_2$.  So $\Gamma^j$ can be
thought as an operation which transfers representations of $\wtG$.

Let $\theta^{p,q}$ be the theta lifting map from $\wtG'_{p,q}$ to $\wtG$.
In the dual pair $(G,G'_{p,q})$, the double covering $\wtG'_{p,q}$ is
split: $\wtG'_{p,q}\cong \bZ/2\bZ \times G'_{p,q}$. So there is a
canonical genuine character $\varsigma$ of $\wtG'_{p,q}$ whose
restriction on $G'_{p,q}$ part is trivial.  By twisting with
$\varsigma$, we identify $G'_{p,q}$-module with
$\wtG'_{p,q}$-module. We abuse notation and denote
$\theta^{p,q}(\rho\otimes \varsigma)$ by $\theta^{p,q}(\rho)$.

We consider characters of $G'_{p,q}$ whose restriction on its Lie
algebra is trivial.  In type $A$ and type $D$, trivial representation
is the only one.  In type $C$, $G'_{p,q} = \rO(p,q)$. Let
$\bfone^{\xi,\eta}$ be the character of $\rO(p,q)$ such that
$\bfone^{\xi,\eta}|_{\rO(p)\times \rO(q)} = \det_{\rO(p)}^{\xi}\otimes
\det_{\rO(q)}^{\eta}$. When $p,q \neq 0$, i.e. $\rO(p,q)$ is
non-compact, all of the four characters of $\rO(p,q)$ are represented
by $\bfone^{\xi,\eta}$ with $\xi,\eta \in \bZ/2\bZ$. When one of $p,q$
is zero, i.e.  $\rO(p,q)$ is compact, there are two characters: the
trivial representation and the determinant.

\begin{thm}\label{thm:e1}
  Fix positive integers $n,m,r,s$ such that $r+s=n$.  Let $(G,
  G'_{m,0})$ be in the stable range with $G'_{m,0}$ the smaller
  member. For type $A$ and type $D$, let $\rho$ and $\rho_{p,q}$ be
  trivial representations; for type $C$, let $\rho = \det_{\rO(m)}^\epsilon$ and
  $\rho_{p,q} = \bfone^{\xi,\eta}$ where $\xi \equiv \epsilon -(s-q) \pmod
  2$ and $\eta \equiv \epsilon -(r-p) \pmod 2$.  Then
  \begin{equation}\label{eq:e1}
    \Gamma^j\theta^{m,0}(\rho) = \bigoplus_{\substack{j = j(p,q)\\
        p+q=m}} \theta^{p,q}(\rho_{p,q}),
  \end{equation}
  Here $0\leq p\leq r, 0\leq q\leq s$ in type~$A$ and type $C$;
  $0\leq 2p\leq r, 0\leq 2q\leq s$ in type~$D$.

\end{thm}

%\subsection{Remarks to Theorem~\ref{thm:e1}}
\subsection{}
The module $\theta^{m,0}(\rho)$ is a singular unitary highest weight
module.  All such modules are classified in \cite{Enright1983}
and they are obtained from theta lifting for classical groups. So
$\theta^{m,0}(\rho) \leadsto \Gamma^j\theta^{m,0}(\rho)$ is a
cohomological operation constructing $\wtG$-modules from these
relatively well understood modules. This type of construction is
studied extensively in \cite{Enright,
  Enright1985,Frajria1991,Wallach1994Trans,WallachZhu2004}.  
%When the $G'_{m,0}$-module $\rho$ is trivial, 
Frajria~\cite{Frajria1991}
studied $\Gamma^j\theta^{m,0}(\rho)$ at the first non-vanishing
degree. He expects these modules could be fit in the dual pair
correspondence.  Later, Wallach and Zhu made a precise
conjecture~\cite{WallachZhu2004}*{Conjecture~5.1} for type $C$ and
they show that \eqref{eq:e1} holds on $K$\mbox{-}spectrum level.  In these works,
the irreducibility and unitarizablity of the resulting modules is the
main concern.  Here the unitarizablity of $\Gamma^j\theta^{m,0}(\rho)$
could follow from Li's result~\cite{Li1989} on the unitarizablity of
stable range theta lifts by \Cref{thm:e1}. Moreover
$\Gamma^j\Theta^{m,0}(\rho)$ could be reducible when there are
(precisely) two pairs $(p_1,q_1)$ and $(p_2,q_2)$ such that $j =
j(p_1,q_1) = j(p_2,q_2)$.

If we consider the derived functor $\Gamma^j$ in all degrees together
, there is a simple formula:
\[
\bigoplus_{j\in \bN}\Gamma^j \theta^{m,0}(\rho) \cong
\bigoplus_{p+q=m}\theta^{p,q}(\rho_{p,q}).
\]
It would be nice if one could explain above formula in terms of some
Euler characteristic formula in the Grothendieck group by adding some
$\pm$-signs. However, we have no idea how to do it yet.

\begin{proof}[Sketch of the proof of \Cref{thm:e1}]
Let $H'$ be the centralizer of $H$ in the real symplectic group
containing the pair $(G,G'_{m,0})$. Now $(H,H')$ is a compact dual
pair, the Fock space $\sY$ is already decomposing into a direct sum of
irreducible unitarizable  $(\fhh,\wtM)$-modules.  So it is easy to
see that $\theta^{m,0}(\rho)$ has the same property
% into irreducible unitarizable $(\fhh,\wtM)$-modules
(see \cite{WallachZhu2004} or \cite[Lemma~49]{MaT}).

A see-saw pair argument gives
\begin{equation*}%\label{eq:dec1}
  \theta^{m,0}(\rho)|_{\fhh,\wtM} = \bigoplus_{\mu\in \cR(\wtH';\sY)} n_\mu
  L(\mu).
\end{equation*}
Here $L(\mu)$ is the theta lift of the $\wtH'$-module $\mu$; $n_\mu
=\dim \Hom_{\wtG'}(\mu,\rho)$ is the multiplicity of $L(\mu)$ occur in
$\theta^{m,0}(\rho)$. In fact, the decomposition is multiplicity free,
since $G'_{m,0}$ is a symmetric subgroup of $H'$.

%  As
% $\wtK_2$-module,
% \[
% \Gamma^j\theta^{m,0}(\rho) = \bigoplus_{\mu\in \cR(\wtH';\sY)} n_\mu
% \Gamma^jL(\mu).
% \]

Note that $L(\mu)$ is irreducible and unitarizable. We apply the decomposition in \Cref{sec:dec}: as
$(\fgg,\wtK_2)$-module,
\begin{equation*}%\label{eq:dec.e1}
  \Gamma^j\theta^{m,0}(\rho) = \bigoplus_{\cW} \Gamma_{\cW}
  \theta^{m,0}(\rho).
\end{equation*}

We apply
Vogan-Zuckerman's theory~\cite{VoganZuckerman1984} to calculate
$\Gamma^jL(\mu)$ case by case, see \cite{WallachZhu2004} and
\cite[Section~3.5.1]{MaT} for details.  The calculation shows that
there are $\wtM$-submodules $\cW_{p,q}$ with multiplicity one in
$\bigwedge^{j(p,q)}\fhh/\fmm$ such that as $\wtK_2$-module,
\begin{equation}\label{eq:e1.iso}
  \Gamma_{\cW_{p,q}}\theta^{m,0}(\rho) \cong
  \theta^{p,q}(\rho_{p,q}).
\end{equation}
Moreover, $\Gamma_{\cW}\theta^{m,0}(\rho) =0$ for other $\cW$.

Setting $G'_1 := G'_{m,0}$ and $G'_2 := G'_{p,q}$, we conclude that
\eqref{eq:e1.iso} also holds as $(\fgg,\wtK_2)$-module by
\Cref{thm:scaler}.  This completes the proof.
\end{proof}

\subsection{Transfer of singular unitary representations}

In this section we consider another type of examples in which singular
unitary representations are transfered between different real forms.
\begin{table}[htpb]
  \centering
  \begin{tabular}{c|cc|c|c}
    Type   & $G_{p,q}$  & $G'$           & stable range       & $j_0$        \\
    \hline
    $A$    & $\rU(p,q)$ & $\rU(n_1,n_2)$ & $p,q \geq n_1+n_2$ & $(n_1+n_2)r$ \\
    $C$    & $\Sp(p,q)$ & $\O^*(2n)$     & $p,q\geq n$        & $2nr$        \\
    $D$    & $\rO(p,q)$ & $\Sp(2n,\bR)$  & $\substack{p, q \geq 2n
      \text{ and}          \\\max\set{p,q}>
      2n}$ & $nr$
  \end{tabular}
  \caption{Transfer of singular unitary representations}\label{tab:exB}
\end{table}

See \Cref{tab:exB} for notation.  We fix a real form $G_1 \subset
G_\bC$ with Cartan involution $\sigma_1$ such that $G_1 \cong
G_{p,q}$.  Now fix an involution $\sigma_2$ commuting with $\sigma_1$
such that $H:= G_\bC^{\sigma_2}\cap G_1 \cong G_{p,r}\times
G_{0,q-r}$.  Let $G_2$ be the real form of $G_\bC$ such that
$\sigma_2$ is its Cartan involution.  It is clear that $G_2 \cong
G_{p+r,q-r}$.  We define functor $\Gamma^j$ by \eqref{eq:gamma1}.

In type $A$ and type $D$, we will assume $p+q$ is even. Then the
double covering $\wtG'$ is split, i.e. $\wtG' \cong G'\times \bZ/2\bZ$
with respect to the dual pair $(G_{p,q},G')$.  Fix the genuine
character $\varsigma$ of $\wtG'$ which is nontrivial on $\bZ/2\bZ$ and
trivial on $G'$. We again identify genuine $\wtG'$-modules with
$G'$-modules via twisting of $\varsigma$.  In particular, the trivial
$G'$-module $\bfone$ corresponds to $\varsigma$.  We denote by
$\theta_{p,q}$ the theta lifting map from $\wtG'$ to $\wtG_{p,q}$.

Note that $\theta_{p,q}(\bfone)$ is not a highest weight module except
for the pairs $(\rO(2,q),\Sp(2,\bR))$. The latter situation is
studied in \cite[Section~8]{Enright1985}.

\begin{thm}\label{thm:ex2}
  Fix positive integers $p,q,n,r$ ($n_1,n_2$ for type~$A$,) such that
  $p+q$ is even in type~$A$ and type $D$. We assume that
  $(G_{p,q},G')$ is in the stable range with $G'$ the smaller member
  and $r<q$.
  Let $\theta_{p,q}(\bfone)$ be the theta lift of the trivial
  representation of $G'$.
  \begin{enumerate}[(i)]
  \item\label{item:ex2.1} If $(G_{p+r,q-r},G')$ is outside the stable
    range, $\Gamma^j\theta_{p,q}(\bfone) = 0$ for every $j$.
  \item\label{item:ex2.2} If $(G_{p+r,q-r},G')$ is in the stable
    range,
    \[
    \Gamma^j\theta_{p,q}(\bfone) = \begin{cases}
      \theta_{p+r,q-r}(\bfone) & \text{when } j=j_0\\
      0 & \text{when }j< j_0,
    \end{cases}
    \]
    where $j_0$ is defined in \Cref{tab:exB}.  If $j>j_0$,
    $\Gamma^j\theta_{p,q}(\bfone)$ is a direct sum of several copies
    of $\theta_{p+r,q-r}(\bfone)$.
  \end{enumerate}
\end{thm}
\begin{proof}[Sketch of the proof]  See
  \cite[Section~3.5.2]{MaT} for the details. 
  Here $\theta_{p,q}(\bfone)$ is again a direct sum of irreducible
  unitarizable $(\fhh,\wtM)$-modules by the argument
  in~\cite{Loke2006Howe}. We apply the decomposition in
  \Cref{sec:dec}. For \eqref{item:ex2.1}
  $\Gamma^j\theta_{p,q}(\bfone)=0$ as $\wtK_2$-module already; For
  \eqref{item:ex2.2} there is a unique $\wtM$-type $\cW_0 \subset
  \bigwedge^* \fhh/\fmm$ such that
  $\Gamma_{\cW_0} V_l$ is non-zero for some $l$. It first occurs in
  $\bigwedge^{j_0}\fhh/\fmm$ with multiplicity $1$ and occurs in
  $\bigwedge^j\fhh/\fmm$ with some multiplicities for $j>j_0$.  One
  calculates that
  \[
  \Gamma_{\cW_0} \theta_{p,q}(\bfone) = \theta_{p+r,q-r}(\bfone)
  \]
  as $\wtK_2$-module. Setting $G'_1 =G'_2= G'$ and applying
  \Cref{thm:scaler}, we finish the proof.
\end{proof}

\subsection{}
The above theorem generalize the results of type $D$ in
\cite{LokeMaTang2011} and the proof here is conceptually simpler.  By
the same argument in \cite{LokeMaTang2011}, we extend the theorem
to theta lifts of unitary highest weight modules:
\begin{cor}\label{Cor:ex2}
  Fix integers $p,q,r,s,n$ ($n_1,n_2$ for type~$A$) such that $p+q+s$
  is even in type~$A$ and type $D$. We retain notations in
  \Cref{tab:exB} and assume that $(G_{p,q+s},G')$ is in the stable
  range.  Let $\mu$ be a genuine $\wtG_{s,0}$-module (finite
  dimensional since $\wtG_{s,0}$ is compact).  Let $L(\mu)$ be the
  unitary highest weight $\wtG'$-module lifted from $\mu$.
  \begin{enumerate}[(i)]
  \item If $(G_{p+r, q+s-r},G')$ is not in the stable range,
    $\Gamma^j\theta^{p,q}(L(\mu)) = 0$ for all $j\in \bN$.
  \item If $(G_{p+r,q+s-r},G')$ is in the stable range,
    \[
    \Gamma^j\theta^{p,q}(L(\mu)) =
    \begin{cases}
      \theta^{p+r,q-r}(L(\mu)) & \text{when } j=j_0\\
      0 &\text{when } j< j_0
    \end{cases}
    \]
    where $j_0$ is defined in \Cref{tab:exB}.  If $j>j_0$,
    $\Gamma^j\theta^{p,q}(L(\mu))$ is a direct sum of several copies
    of $\theta^{p+r,q-r}(L(\mu))$.
  \end{enumerate}
\end{cor}

\subsection{}
In view of \Cref{thm:e1}, \Cref{thm:ex2} and \Cref{Cor:ex2}, we
speculate that theta correspondence and derived functors would be
compatible upto Langlands packets or Arthur packets.

\end{document}

%% file: mymacros.tex
\def\Ann{{\mathrm{Ann}\,}}
\def\Ker{{\rm Ker}\,}
\def\Lie{{\rm Lie}}

\def\Hom{{\rm Hom}}
\def\End{{\rm End}}

\def\Ind{{\rm Ind}}

\def\v0{{\vec{0}}}

\def\sl2{{\mathfrak{sl}(2)}}

\def\sp{{\mathfrak{sp}}}
\def\fsp{{\mathfrak{sp}}}

\def\cF{\mathcal{F}}

\def\cI{\mathcal{I}}

\def\cN{\mathcal{N}}

\def\cR{\mathcal{R}}

\def\cT{\mathcal{T}}
\def\cU{\mathcal{U}}

\def\cW{\mathcal{W}}

\def\cZ{\mathcal{Z}}

\def\sC{\mathscr{C}}

\def\sR{\mathscr{R}}

\def\sT{\mathscr{T}}
\def\sU{\mathscr{U}}
\def\sV{\mathscr{V}}

\def\sY{\mathscr{Y}}

\def\bC{{\mathbb{C}}}

\def\bN{{\mathbb{N}}}

\def\bR{{\mathbb{R}}}

\def\bZ{{\mathbb{Z}}}

\def\rG{{\rm{G}}}
\def\rH{{\rm{H}}}

\def\rK{{\rm{K}}}

\def\rM{{\rm{M}}}

\def\rO{{\rm{O}}}

\def\rR{{\rm{R}}}

\def\rU{{\rm{U}}}

\def\fee{\mathfrak{e}}

\def\fgg{\mathfrak{g}}
\def\fhh{\mathfrak{h}}

\def\fkk{\mathfrak{k}}

\def\fmm{\mathfrak{m}}

\def\wtG{\widetilde{G}}
\def\wtH{\widetilde{H}}

\def\wtK{\widetilde{K}}

\def\wtM{\widetilde{M}}

\def\wtSp{\widetilde{\rSp}}

\def\inn#1#2{\left\langle{#1},{#2}\right\rangle}

\def\Sp{{\rm Sp}}
\def\rSp{{\rm Sp}}

\def\O{{\rm O}}
\def\rO{{\rm O}}

\def\det{{\rm det}}

\def\bfone{{\bf 1}}

\def\fso{\mathfrak{so}}

\def\cuu{\mathcal{U}}

\def\czz{\mathcal{Z}}
\def\cug{\cuu(\fgg)}
\def\fmm{\mathfrak{m}}

\def\csigma\
\def\Ind{{\rm Ind}}

\def\seesawpair#1#2#3#4{
\xymatrix{
{#1} \ar@{-}[dr]& {#2} \\
{#3} \ar@{-}[ur] & {#4}}
}

\newtheorem{thmA}{Theorem}

\newtheorem*{thm*}{Theorem}

\ifdefined\frametitle{}
\else{
\ifdefined\thm{}\else
\newtheorem{thm}{Theorem}
\fi
\ifdefined\lemma{}\else
\newtheorem{lemma}[thm]{Lemma}
\fi
\ifdefined\prop{}\else
\newtheorem{prop}[thm]{Proposition}
\fi
\ifdefined\cor{}\else
\newtheorem{cor}[thm]{Corollary}
\fi
\ifdefined\example{}\else
\newtheorem{example}[thm]{Example}
\fi
\ifdefined\conjecture{}\else{
\newtheorem{conjecture}[thm]{Conjecture}}
\fi
\ifdefined\remark{}\else{
\theoremstyle{remark}
\newtheorem{remark}[thm]{Remark}}
\fi
\ifdefined\definition{}\else{
\theoremstyle{definition}
\newtheorem{definition}[thm]{Definition}}
\fi
}
\fi

\long\def\quest #1 {\ifdefined\showquest{\color{blue}\small
#1\normalfont}\fi}

\long\def\ignore#1{\empty}

\makeatletter
\def\saveenum{\xdef\@savedenum{\the\c@enumi\relax}}
\def\resetenum{\global\c@enumi\@savedenum}
\makeatother